\documentclass[11pt]{amsart}

\usepackage{amssymb, graphicx}

\usepackage{amsfonts}
\usepackage{latexsym}
\usepackage{amscd}

\newtheorem{theorem}{Theorem}[section]

\newtheorem{lemma}[theorem]{Lemma}
\newtheorem{corollary}[theorem]{Corollary}

\newcommand{\tr}{\langle \tau \rangle}
\newcommand{\A}{\mathcal{A}}

\newcommand{\wt}[1]{\widetilde{#1}}
\newcommand{\G}{\wt{G}}

\begin{document}
\title[Maps preserving zeros of a polynomial]{Maps preserving zeros of a polynomial}

\author{J. Alaminos}
\author{M. Bre\v sar}
\author{\v S. \v Spenko}
\author{A.\,R. Villena}
\address{J. Alaminos  and A.\,R. Villena, Departamento de An\' alisis
Matem\' atico, Facultad de Ciencias, Universidad de Granada,
 Granada, Spain} \email{alaminos@ugr.es,
avillena@ugr.es}
\address{M. Bre\v sar,  Faculty of Mathematics and Physics,  University of Ljubljana,
 and Faculty of Natural Sciences and Mathematics, University
of Maribor, Slovenia} \email{matej.bresar@fmf.uni-lj.si}
\address{\v S. \v Spenko,  Institute of  Mathematics, Physics, and Mechanics,  Ljubljana, Slovenia} \email{spela.spenko@imfm.si}

\thanks{
2010 {\em Math. Subj. Class.} 15A86, 16R10, 16R60, 20G15, 47B48.}
\thanks{The first and the fourth named authors were supported
by MICINN  Grant MTM2009-07498
and Junta de Andaluc\'ia Grants FQM-185 and
P09-FQM-4911.
The second named author  was supported by ARRS Grant P1-0288.}

\begin{abstract}
Let $\A$ be an algebra and let $f(x_1,\ldots,x_d)$ be a multilinear polynomial in noncommuting indeterminates $x_i$.
We consider the problem of describing linear maps $\phi:\A\to \A$ that preserve zeros of $f$. Under certain technical restrictions we solve the problem for general polynomials $f$ in the case where $\A=M_n(F)$. We also consider quite  general algebras $\A$, but only for specific polynomials $f$.
\end{abstract}
\maketitle

\section{Introduction}

Let $F$ be a field, let $F\langle X\rangle$ be the free algebra generated by the set  $X=\{x_1,x_2,\ldots\}$ of countably
many noncommuting indeterminates, and let  $f=f(x_1,\ldots,x_d)\in F\langle X\rangle$ be a nonzero polynomial.
We say that a map $\phi$ from an $F$-algebra $\A$ into itself  {\em preserves zeros of $f$} if   for all $a_1,\ldots,a_d\in \A$,
 $$f(a_1,\dots,a_d)=0\,\,\Longrightarrow \,\, f(\phi(a_1),\dots,\phi(a_d))=0.$$
 The list of all maps on $\A$ that preserve zeros of $f$ must  certainly contain scalar multiples  of automorphisms,
 for some polynomials   it must also contain
scalar multiples of antiautomorphisms (say, for $f=x_1x_2+x_2x_1$), and for some of them    even   all maps of the form
\begin{equation}\label{sf}
 \phi(x) = \alpha \theta(x) + \mu(x),
 \end{equation}
 where $\alpha\in F$, $\theta:\A\to \A$ is either an automorphism or an antiautomorphism, and
 $\mu$ is a linear map from $\A$ into its center (say, for $f=x_1x_2-x_2x_1$).
 Our goal is to show that under certain   restrictions - in particular, we will  confine ourselves to linear maps
 $\phi$ and multilinear polynomials $f$ - the standard example \eqref{sf} is also the only possible example of a map preserving zeros of $f$. We will not bother with the
 question for which polynomials \eqref{sf} can be  simplified.

 For certain simple  polynomials, especially for $f=x_1x_2$ and $f=x_1x_2-x_2x_1$, our problem has a long and rich history; see, for example, \cite{ABEV0} and \cite{FI} for historic comments and references.  So far not much is known  for general polynomials. For them the problem was explicitly posed by Chebotar et al. \cite{CFL} for the matrix algebra   $\A=M_n(F)$, and some partial solutions were obtained in two recent papers:  \cite{GK}  considers, in particular, the case where the sum of coefficients of $f$ is a nonzero scalar (without assuming the linearity of $\phi$), and
  \cite{DD} handles Lie polynomials of degree at most $ 4$. Let us also mention a related, yet considerably simpler, problem of describing linear maps that preserve all values of $f$, i.e.,
   $\phi(f(a_1,\dots,a_d)) =f(\phi(a_1),\dots,\phi(a_d))$ for all $a_i\in\A$.
  This problem can be solved at a high level of generality by using functional identities, although for finite dimensional algebras (including $M_n(F)$) the obtained results are not optimal; see \cite{BF} and also \cite[Section 6.5]{FI}.

One of the most fascinating approaches to linear preserver problems on matrix algebras  was developed by Platonov and \DJ okovi\' c in \cite{PD}. It is based on linear algebraic groups. In Section 2 we will see that this approach is applicable to our problem. In  the matrix algebra $\A=M_n(F)$ we will be able to consider general multilinear polynomials $f$; however, we will be forced to impose  several technical restrictions some of which might be superfluous. The general problem from \cite{CFL} is therefore not yet completely solved.

In Section 3 we will prove three results giving solutions to our problem for some special polynomials, but in the context of rather general classes of prime algebras and/or  $C^*$-algebras. More precisely, we will show that for these polynomials the problem can be reduced to some still  nontrivial,  but already solved problems. For polynomials that are not covered in our considerations, or at least cannot be handled by similar methods,  the problem seems to be very intriguing.

% ...  In the proof of this theorem we will take the liberty of using numerous arguments from \cite{PD}, especially from the proof of Theorem 1 from this paper which deals with the polynomial $x_1x_2-x_2x_1$. However, dealing with a general polynomial does make the problem more demanding, and some methods that are not included in \cite{PD} also have to be used.

\section{The matrix algebra case}

The main goal of this section is to prove Theorem \ref{7}. First we will survey the necessary tools  needed in the proof.

\subsection{Remarks on free algebras}

Let $F$ be a field and
let $X=\{x_1,x_2,\ldots\}$ be a set of countably many noncommuting indeterminates.
The free algebra $F\langle X\rangle$ consists of polynomials in $x_1,x_2,\ldots$.
We say that  $f=f(x_1,\ldots,x_d)\in F\langle X\rangle$ is a  {\em multilinear polynomial}  if it is of the form
\begin{equation*}%\label{poly}
f= \sum_{\sigma\in S_d} \lambda_\sigma x_{\sigma(1)}\ldots  x_{\sigma(d)},
\end{equation*}
where $ \lambda_\sigma\in F$ and  $S_d$ is the symmetric group of degree $d$. A nonzero polynomial
 $f=f(x_1,\ldots,x_d)\in F\langle X\rangle$ is said to be a {\em polynomial identity} of an
$F$-algebra $\A$ if $f(a_1,\ldots,a_n) =0$ for all $a_1,\ldots,a_n\in  \A$. For example,  $\A$ is a commutative algebra  if and only if $[x_1,x_2] = x_1x_2-x_2x_1$ is its polynomial identity. By the famous Amitsur-Levitzki theorem, the matrix algebra $M_n(F)$ has a polynomial identity of degree $2n$. On the other hand, $M_n(F)$ does not have  polynomial identities of degree $ < 2n$; the proof of that will be used in our arguing.

Let $F\langle X\rangle_0$ denote  the  subalgebra  of  $F\langle X\rangle$ generated by $1$ and all polynomials of the form  $[x_{k_1},[x_{k_2},\dots , [x_{k_{r-1}},x_{k_r}]\dots]]$. That is to say, $F\langle X\rangle_0$ is the subalgebra  generated by $1$ and all Lie polynomials of degree $\ge 2$. Defining the partial derivative  $\frac{\partial f}{\partial x_i}$ of $f\in F\langle X\rangle$ in a self-explanatory manner it is easy to see that $\frac{\partial f}{\partial x_i}$ is always $0$ if $f\in  F\langle X\rangle_0$. Moreover, if char$(F)=0$, then this property is characteristic for elements from $F\langle X\rangle_0$ \cite[Proposition 3]{Ger}. Note that if $f = f(x_1,\ldots,x_d)$ is a multilinear polynomial, its partial derivative can be simply obtained by formally replacing $x_i$ by $1$:  $$\frac{\partial f}{\partial x_i} = f(x_1,\ldots,x_{i-1},1,x_{i+1},\ldots,x_d).$$

\subsection{The Platonov-\DJ okovi\' c theory }
Let $K$ be an algebraically closed field of characteristic 0. We will write  $M_n$ for  $M_n(K)$.  We have
$M_n =M_n^0 \oplus K\cdot 1$,
where $1$ is the identity matrix, and $M_n^0$ is the space of all $x\in M_n$ with ${\rm tr}(x) =0$.
%By $GL(n^2 - 1)$ we denote the subgroup of $GL(n^2)$ which fixes $1$ and preserves
%$M_n^0$. %Note that G -~ PSL(n).
%Further,
Let $O(n^2)$ be the subgroup of $GL(n^2)$ which preserves the nondegenerate
symmetric bilinear form ${\rm tr}(xy),\; x,y\in M_n$.
The subgroup
of $O(n^2)$ consisting of operators which fix the identity matrix $1$ will be denoted by $O(n^2 - 1)$.
 % Since the decomposition \eqref{1.1} is orthogonal with respect to this form, $O(n^2- 1)$ preserves $M_n^0.$
The identity components of $O(n^2)$ and $O(n^2 - 1)$, i.e., subgroups consisting of matrices whose determinant is $1$,
%(i.e.,  the connected components with respect to Zariski topology that contain the identity)
 will be denoted
by $SO(n^2)$ and $SO(n^2 - 1)$, respectively.

 By $G$ we denote  the subgroup of $GL(n^2)$   consisting of all similarity transformations
$x\mapsto axa^{-1}$ with $a\in GL(n)$.
%As $G$ also preserves the form
%${\rm tr}(xy)$ and fixes $1$, we have $G\subseteq SO(n^2 - 1)$.
%Note that the transposition map $\tau$ lies in $O(n^2 - 1)$ and that $\det(\tau) = (-1)^{ n(n-1)/2}$.
Next, by $P$  we denote the subgroup of $GL(n^2)$ which acts
trivially on $M_n^0$ and $M_n/M_n^0$, and by $Q$  the subgroup of
$GL(n^2)$ which acts trivially on  $K1$ and $M_n/K1$. Thus, $Q$
consists of  all transformations $x\mapsto x + f(x)1$, where $f$
is a linear functional on $M_n$ such that $f(1)=0$.  Let $T$
denote the subgroup of $GL(n^2)$ which acts by scalar
transformations on $M_n^0$ and $K1$, and  set 
$T_1 = T \cap SL(n^2).$

By $\tau$ we denote the transposition map. However, we will  write $x'$ for the transpose of $x$.
Note that the group $GQT\tr$
consists of all invertible linear transformations $\sigma:M_n\to M_n$ that take one of the forms
$
\sigma(x) =\alpha axa^{-1} + f(x)1$ or
$
\sigma(x) = \alpha ax'a^{-1} + f(x)1,
$
where $\alpha\in K^*$, $a\in GL(n),$ and $f$ is a linear functional on $M_n$ such that
$f(1)\neq -\alpha.$

The algebra of all linear transformations on $M_n$ can be identified with the tensor product algebra $M_n \otimes M_n^{opp}$, where $ M_n^{opp}$ is the opposite algebra of $M_n$, via the action
$(a\otimes b)(x) = axb$, $a,b,x\in M_n$.

With respect to the notations just introduced, the following theorem can be extracted from \cite[Theorems A and B]{PD}.

\begin{theorem}\label{izrekA} {\bf (Platonov-\DJ okovi\' c)}
Let $ \Gamma$ be a proper connected algebraic subgroup of $SL(n^2)$, $n\ne 4$, containing
$G$. Then $\Gamma$ is one of the groups:
\begin{enumerate}
 \item[(a)] $G$, $GQ$, $GT_1$,  $GQT_1$,
 \item[(b)] $SO(n^2 - 1)$, $SO(n^2 - 1)T_1$, $SO(n^2 - 1)P$, $SO(n^2 - 1)Q$, \newline $SO(n^2 - 1)PT_1$, $SO(n^2 - 1)QT_1$, $SL(n^2 - 1)$, $SL(n^2 - 1)T_1$,\newline $SL(n^2 - 1)P$, $SL(n^2 - 1)Q$, $SL(n^2 - 1)PT_1$, $SL(n^2 - 1)QT_1$,\newline $tSO(n^2)t^{-1}$ for some $t\in T_1$,
 \item[(c)] $GP$, $GPT_1$,
 \item[(d)] $t\Bigl(SL(n)\otimes SL(n)^{opp}\Bigr)t^{-1}$ for some $t\in T_1$.
\end{enumerate}
Moreover, if $\Gamma$ is one of the  groups listed in {\rm (a)}, then its normalizer in $GL(n^2)$ is a subgroup of $GQT\langle \tau \rangle.$
\end{theorem}
Let us point out that all groups listed in (b) contain $SO(n^2-1)$. For $tSO(n^2)t^{-1}$, $t\in T_1$, this can be easily checked, while for others this is entirely obvious. Conversely, only the groups from (b) contain $SO(n^2-1)$.

%\begin{theorem}\label{izrekA}
%Let $ \Gamma$ be a proper connected algebraic subgroup of $SL(n^2)$ containing
%$G,$ $n\neq 4.$ Then

%(i) if contains $P$ or $Q$, $ \Gamma$ is one of the groups
%\begin{equation}\label{1.2}
%HP, HQ, HPT_1, HQT_1,
%\end{equation}

%(ii) or otherwise $\Gamma$ is one of the groups

%\begin{equation}
%tSO(n^2)t^{-1}, t(SL(n)\otimes SL(n)^{opp})t^{-1}, H, HT_1,
%\end{equation}

%where $H \in \{ G, SO(n^2 - 1), SL(n^2 - 1 )\}$, $t\in T_1.$

%\end{theorem}

%We also need an additional result from the same article \cite{PD}, that the normalizer of the group $GQT_1$ in $GL(n^2)$ is equal to $GQT\langle \tau \rangle.$

\subsection{Main theorem}

Let $f=f(x_1,\ldots,x_d)\in F\langle X\rangle$ be a nonzero multilinear polynomial of degree $d$. Our goal is to show that under suitable assumptions a  linear map $\phi:M_n(F)\to M_n(F)$ that preserves zeros of $f$  is of the standard form \eqref{sf}. In the present  setting this can be  more specifically described as
\begin{equation}\label{std0}
\phi(x) =\alpha axa^{-1} + f(x)1
\;\;\mbox{or}\;\;
\phi(x) = \alpha ax'a^{-1} + f(x)1,
\end{equation}
where $\alpha\in F^*$, $a\in GL(n,F),$ and $f$ is a linear functional on $M_n(F)$ such that
$f(1)\neq -\alpha.$

 If $d$ was $\ge 2n$, then, by the Amitsur-Levitzki theorem, $f$ could be  a polynomial identity, making the  assumption that $\phi$ preserves zeros of $f$  meaningless. We will therefore assume that $d < 2n$.  Further, we will assume that $n\ne 2,4$. It is well-known that the $n=2$ case must be excluded when dealing with the polynomial $x_1x_2 - x_2x_1$. On the other hand, it seems possible that that the exclusion of $n=4$ is unnecessary. We need it in order to apply Theorem \ref{izrekA}.
 %Let us remark that our theorem is trivially true if $n=1$, but in the proof we will tacitly assume that $n\ne 1$.
Another assumption that we have to require is that char$(F)=0$. This one is  also used because of applying Theorem \ref{izrekA} and is possibly redundant. Further, we will assume that $\phi$ is bijective. This is a usual and certainly necessary assumption in this context (cf. \cite{BrSe} that deals with the polynomial $x_1x_2 - x_2x_1$ without assuming bijectivity). Finally, we will assume that $\phi(1)\in F\cdot 1$; the (un)necessity of this assumption will be discussed in the next subsection.

Let us make a few  comments and introduce some notations before stating and proving the theorem. 
We have to warn the reader that, just as in \cite{PD}, we are assuming a basic familiarity with the concepts related to linear algebraic groups. A good general reference is  Borel's book \cite{Bor}.

We are going to consider a bijective linear map $\phi$ on $M_n(F)$ that preserves the
  set of zeros of  $f$,
$$S_F=\{(a_1,\dots,a_d)\in M_n(F)^d\;|f(a_1,\dots,a_d)=0\}.$$
This is an algebraic set. Indeed, considering $S_F$ as  a subset of  $(F^{n^2})^d$, it  is equal to the
 vanishing set
of polynomials $\{f((x^1_{ij}),\dots,(x^d_{ij}))_{st}|\;1\leq s,t\leq n\}$. Using \cite[Lemma 3]{PD} (or \cite[Lemma 1]{Dix}) it can be therefore deduced that $\phi^{-1}$ also preserves $S_F$, i.e., it also satisfies the condition we are interested in (and so, in fact, $\phi(S_F) = S_F$). Accordingly, the set of all linear maps satisfying this condition is an algebraic group.  The goal of our theorem is to describe those of its elements that also  preserve scalar matrices.

By $K$ we denote an algebraic closure of $F.$ Since $\phi\in GL(n^2,F)\subseteq GL(n^2)\, (=GL(n^2,K))$ preserves $S_F$, it also preserves  its Zariski closure $S$ in $K^{n^2d}.$ This is an algebraic set, and therefore, by the same argument as above,  $$\G=\{\psi\in GL(n^2,K)|\; \psi(S)\subseteq S\}$$
is  a group (and  $\psi(S)= S$ for every $\psi\in\G$). By $M$ we denote the (algebraic) subgroup of $GL(n^2)$ consisting of all maps that preserve scalar matrices. Thus $\phi$ is contained in the algebraic group $\G\cap M$.

For every algebraic group $L$ defined over $F$ we denote by $L_F$ the group of $F$-rational points of $L.$ We have $(GQT)_F=G_FQ_FT_F$ and $(GQT\tr)_F$ consists of  elements in $ GL(n^2,F)$ that are of the form \eqref{std0}; cf. \cite[p 176]{PD}. Thus, if one can establish that
\begin{equation}\label{cilj}
\G\cap M\subseteq GQT\tr,
\end{equation}
then $\phi$, which is defined over $F$, lies in $ G_FQ_FT_F\tr$ and is therefore of the standard form \eqref{std0}.

Note that $S_F$ is invariant under the $G_F$-action given by $$g\cdot(a_1,\dots,a_d):=(g(a_1),\dots,g(a_d)).$$ Hence its closure $S$ is also invariant under $G_F,$  so that
$G_F\subseteq \G$.
Since char$(F)=0$  and  $G$ is connected, the rational points  $G_F$ are Zariski-dense in $G$  \cite[Corollary 18.3]{Bor}. From this one infers that $G =\overline{G_F}\subseteq \G$; moreover, $G\subseteq \G\cap M$.
In a similar fashion, by first noticing that $S_F$ is closed under multiplication by nonzero scalars in $F$ we see that
$\G\cap M$ is closed under  multiplication by nonzero scalars in $K$; that is, if $a\in\G$ and $\lambda\in K^*$, then $\lambda a \in \G\cap M$.

Let us also mention that if $H$ is an arbitrary algebraic group, then its identity component
(i.e., the connected component with respect to Zariski topology that contains the identity)
 is also an algebraic group, and moreover, it is a normal subgroup of $H$ \cite[Proposition 1.2]{Bor}.

We now have enough information to prove the following theorem.

\begin{theorem}\label{7}
Let $F$ be a field with {\rm char}$(F) =0$, let $f\in F\langle X\rangle$ be a  multilinear polynomial of degree $d\ge 2$, and let $\phi:M_n(F)\to M_n(F)$ be a bijective linear map that preserves zeros of $f$ and satisfies $\phi(1)\in F\cdot 1$. Assume that  $n\neq 2, 4$ and $d< 2n$.
Then $\phi$ is of the standard form \eqref{std0}.
\end{theorem}

\begin{proof}
As noticed above, it suffices to establish  \eqref{cilj}. We claim that it is enough to prove
\begin{align}
SO(n^2-1)\not \subseteq \wt{G}.\label{so}
%P \not \subseteq \G,\label{P}\\
% t\Bigl(SL(n)\otimes SL(n)^{opp}\Bigr)t^{-1}  \label{tenz}\not \subseteq \G.
\end{align}
Indeed, assume \eqref{so} holds. Consider  $H = (\G\cap M)\cap SL(n^2)$ and let $H_1$ be the identity component of $H$. Then $H_1$ is an algebraic group, it is connected, and, since $G\subseteq \G\cap M$, it contains $G$. Therefore $H_1$ is one of the groups listed in Theorem \ref{izrekA}. As $H_1\subseteq \G\cap M$ and \eqref{so} holds, we may exclude the possibilities listed in (b). Furthermore,
as the groups from (c) and (d) are not contained in $M$,
 $H_1$ must be one of the groups listed in (a). Theorem \ref{izrekA} now tells us that the  normalizer of $H_1$ in $GL(n^2)$ is a subgroup of $GQT\langle \tau \rangle$. Since $H_1$ is a normal subgroup of $H$ it follows that $H$ is contained in $GQT\langle \tau \rangle$. Now pick $\alpha\in \G\cap M$. As mentioned above, $\G\cap M$ is closed under multiplication by nonzero scalars. Therefore $\det(\alpha)^{-1}\alpha\in H \subseteq  GQT\langle \tau \rangle$. As $GQT\langle \tau \rangle$ is also closed under multiplication by nonzero scalars it follows that $\alpha = \det(\alpha)\bigl(\det(\alpha)^{-1}\alpha\bigr) \in GQT\langle \tau \rangle$. This proves \eqref{cilj}.

Thus, let us prove \eqref{so}. Assume first that $d$ is an even number.
 Set $k=\frac{d}{2} + 1$ and note that $k \le n$. Consider the sequence of $d$ matrix units
\begin{equation}\label{enote}
e_{11},\,e_{12},\,e_{22},\,e_{23},\,\,e_{33},\,e_{34},\dots,e_{k-1,k-1},\, e_{k-1,k}.
\end{equation}
The product of these matrices in an arbitrary order except in the given one is equal to zero. Therefore, for an appropriate permutation  $(a_1,\dots,a_{d})$ of the  matrices (\ref{enote}) (corresponding to a nonzero coefficient of $f$) we have $f(a_1,\dots, a_{d})\neq 0$. Now define a linear transformation $\theta$ on $M_n(K)$ according to
$$\theta(e_{12}) = e_{21},\,\,\theta(e_{21}) = e_{12},\,\,\theta(e_{11}) = e_{33},\,\,\theta(e_{33}) = e_{11},$$
%$$\theta(e_{12}) = e_{13},\,\,\theta(e_{21}) = e_{31},\,\,\theta(e_{13}) = e_{12},\,\,\theta(e_{31}) = e_{21},$$
 and $\theta$ fixes all other matrix units. A bit tedious but straightforward verification shows that $\theta$ lies in $SO(n^2-1)$. Now, $\theta$ maps the matrices from \eqref{enote} into the matrices
 $$e_{33},\,e_{21},\,e_{22},\,e_{23},\,e_{11},\,e_{34},\dots,e_{k-1,k-1},\,e_{k-1,k}.$$
Their product  in an arbitrary order is  $0$, so that  $f(\theta (a_1),\dots,\theta(a_{d}))=0$. This implies that $\theta \not\in \G$. Namely, if $\theta$ was in $\G$ then   $\theta^{-1}$
would map $S_F$ into $S$ which is contained in the set of zeros of $f$. Thus \eqref{so} is proved in this case.

The case where $d$ is odd requires only minor modifications. One has to consider the matrix units
$$
e_{11},\,e_{12},\,e_{22},\,e_{23},\,\dots,e_{k-1,k},\,e_{k,k},
$$
where $k = \frac{d+1}{2}\le n$, and then follow the above argument.
\end{proof}

\subsection{Preserving scalar matrices} It seems plausible that the assumption from Theorem \ref{7} that  $\phi(1)\in F\cdot 1$ can be removed.
To this end one should examine carefully the groups from (c) and (d).
However, apparently  this would require a detailed and tedious analysis  making the proof much lengthier. We have therefore decided  to omit this problem in its full generality here, and perhaps return to it in a more technical paper. We will now restrict our attention to polynomials from $F\langle X\rangle_0$, which are of special interest in view of \cite{DD}.  For these polynomials the argument based on
the Platonov-\DJ okovi\' c theory is rather short. However, we will use an alternative approach, based on the following elementary lemma which is perhaps of independent interest.

\begin{lemma} \label{el}
Let $f\in F\langle X\rangle$, where $F$ is an arbitrary field, be a  multilinear polynomial of degree $d$. Let $n\ge 2$ be such that  $d < 2n$. If $c\in M_n(F)$ satisfies
\begin{equation} \label{ele}
f(c,a_2,\ldots,a_d) = f(a_1,c,a_3,\ldots,a_d) = \ldots = f(a_1,\ldots,a_{d-1},c) = 0
\end{equation}
for all $a_1,\ldots,a_d\in M_n(F)$, then $c\in F\cdot 1$.
\end{lemma}

\begin{proof}
Pick an arbitrary  rank one idempotent $e\in M_n(F)$. Then the algebra $(1-e)M_n(F)(1-e)$ is isomorphic to $M_{n-1}(F)$, so it contains matrix units $h_{ij}$, $1\le i,j \le n-1$, i.e., elements satisfying $h_{ij}h_{kl} = \delta_{jk}h_{il}$ and $\sum_{k=1}^{n-1} h_{kk} = 1-e$.

Without loss of generality we may assume that $x_1x_2\ldots x_d$ is a monomial of $f$. We set $(s,t) := (\frac{d}{2}-1,\frac{d}{2})$ if $d$ is even and  $(s,t) := (\frac{d-1}{2},\frac{d-1}{2})$ if $d$ is odd. In any case we have $t\le {n-1}$. Examining all possible monomials of $f$ one easily notices that
$$
e\cdot f(e,c,h_{11},h_{12},h_{22},h_{23},\ldots,h_{st})\cdot h_{t1} = ech_{11}.
$$
Since $f(e,c,h_{11},h_{12},h_{22},h_{23},\ldots,h_{st})=0$ by our assumption, we thus have $ech_{11}=0$. Similarly, by permuting the $h_{ij}$'s, we see that $ech_{kk}=0$ for every $k$. Accordingly, $ec(1-e)=0$. In a similar fashion, by using $f(a_1,\ldots,a_{d-2},c,a_d) = 0$, we get $(1-e)ce=0$. Hence it follows that $c$ commutes with every rank one idempotent $e$. But then $c\in F\cdot 1$.
\end{proof}

\begin{corollary}\label{70}
Let $F$ be a field with {\rm char}$(F) =0$, let $f\in F\langle X\rangle_0$ be a multilinear polynomial of degree $d\ge 2$, and let $\phi:M_n(F)\to M_n(F)$ be a bijective linear map that preserves zeros of $f$. Assume that  $n\neq 2, 4$ and $d< 2n$.
Then $\phi$ is of the standard form \eqref{std0}.
\end{corollary}

\begin{proof}
In view of Theorem \ref{7} it suffices to prove that $c:=\phi(1)$ lies in $F\cdot 1$. This is an immediate consequence of Lemma \ref{el}. Namely, since $f\in F\langle X\rangle_0$ we have
$$
f(1,b_2,\ldots,b_d) = f(b_1,1,b_3,\ldots,b_d) = \ldots = f(b_1,\ldots,b_{d-1},1) = 0
$$
for all $b_i\in M_n(F)$, and hence  \eqref{ele} follows.
\end{proof}

\section{Some special polynomials}

In this section we will consider some special multilinear polynomials 
\begin{equation}\label{ef}
f(x_1,x_2,\ldots,x_d) = \sum_{\sigma\in S_d}\lambda_{\sigma}x_{\sigma(1)}x_{\sigma(2)}\ldots x_{\sigma(d)}
\end{equation}
 for which our problem can be handled  in rather general classes of algebras.
% More precisely, for these polynomials the problem can be reduced to some still  nontrivial,  but already solved problems.
Specifically, we will consider polynomials $f$ satisfying one of the following conditions:
\begin{equation}\tag{A} \label{pd1}
\frac{\partial^{d-1} f}{\partial x_{2}\partial x_3\ldots \partial x_{d} }\ne 0,
\end{equation}
\begin{equation}\tag{B} \label{pd2}
\frac{\partial^{d-1} f}{\partial x_{2}\partial x_3\ldots \partial x_{d} }= 0 \quad\mbox{and}\quad
\frac{\partial^{d-2} f}{\partial x_{3}\partial x_4\ldots \partial x_{d} }\ne 0,
\end{equation}
\begin{equation}\tag{C}\label{pd3}
\frac{\partial^{d-1} f}{\partial x_{2}\partial x_3\ldots \partial x_{d} }= 0 \quad\mbox{and}\quad
f(x,x,\ldots,x,y)\ne 0.
\end{equation}
The conditions \eqref{pd2} and \eqref{pd3} are independent. For example,
$$
x_1x_2x_3 + x_3x_1x_2 + x_2x_3x_1 - x_2x_1x_3 - x_1x_3x_2 - x_3x_2x_1
$$
(i.e., the standard polynomial of degree $3$) satisfies  \eqref{pd2} and does not satisfy \eqref{pd3}, while
$$
x_1(x_2x_3 - x_3x_2) - (x_2x_3 - x_3x_2)x_1
$$
satisfies  \eqref{pd3} and does not satisfy \eqref{pd2}.

\subsection{Polynomials satisfying \eqref{pd1}}
We begin with an elementary lemma.

\begin{lemma}\label{L31}
Let $f$ be a multilinear polynomial satisfying \eqref{pd1}. Suppose $\A$ is a unital algebra and   $\phi:\A\to \A$ is a linear map  preserving zeros of $f$ and satisfying
 $\phi(1)\in F^*1$. 
If $a,b\in A$ are such that $ab=ba=0$, then $\phi(a)\phi(b) + \phi(b)\phi(a) =0$.
  \end{lemma}

\begin{proof} Without loss of generality we may assume that $\phi(1)=1$. Namely, if $\phi(1)= \lambda 1$ with  $0\ne\lambda\in F$, then we can replace $\phi$ by $\lambda^{-1}\phi$ which also 
preserves zeros of $f$ and  does map $1$ into $1$.

 From $ab=ba=0$ we infer
$$f(a,b,1,\ldots,1) =
f(b,a,1,\ldots,1) =0,$$
 and hence
$$f(\phi(a),\phi(b),1,\ldots,1) =
f(\phi(b),\phi(a),1,\ldots,1) =0.$$
We write $f$ as in \eqref{ef}. Note that  \eqref{pd1} simply means that
\begin{equation*}
\lambda := \sum_{\sigma\in S_d}\lambda_{\sigma} \ne 0.
\end{equation*}
Since
\begin{align*}
&f(\phi(a),\phi(b),1,\ldots,1) +
f(\phi(b),\phi(a),1,\ldots,1)\\
 = &\sum_{\sigma^{-1}(1) < \sigma^{-1}(2)}\lambda_\sigma \phi(a)\phi(b) + \sum_{\sigma^{-1}(2) < \sigma^{-1}(1)}\lambda_\sigma \phi(b)\phi(a) \\
 +& \sum_{\sigma^{-1}(1) < \sigma^{-1}(2)}\lambda_\sigma \phi(b)\phi(a) + \sum_{\sigma^{-1}(2) < \sigma^{-1}(1)}\lambda_\sigma \phi(a)\phi(b) \\
 =& \lambda\Bigl((\phi(a)\phi(b) + \phi(b)\phi(a) \Bigr),
\end{align*}
it follows that $\phi(a)\phi(b) + \phi(b)\phi(a) =0$.
\end{proof}

%The conclusion of Lemma \ref{L31} makes it possible for us to directly apply a result from \cite{ABEV}. 
Recall that a {\em Jordan epimorphism} on an algebra $\A$ is a  surjective linear map $\theta$ satisfying $\theta(a^2) =\theta(a)^2$ for every $a\in \A$. 

\begin{theorem} \label{Gra1}
Let $f$ be a multilinear polynomial of degree $d\ge 2$ satisfying \eqref{pd1}, and let
 $\A$ be a unital $C^*$-algebra.  If a continuous
surjective linear map $\phi:\A\to\A$ preserves zeros of $f$ and satisfies $\phi(1)\in \mathbb C^* \cdot 1$,
then $\phi$ is a scalar
multiple of a Jordan epimorphism.
\end{theorem}

\begin{proof}
The conclusion of Lemma \ref{L31} makes it possible for us to directly apply \cite[Theorem 3.3]{ABEV}. The statement of this theorem together with a well-known fact that  Jordan epimorphisms  preserve unities \cite[Corollary 3, p.\,482]{JR} immediately gives the desired conclusion.
\end{proof}

\begin{corollary}
Assume the conditions of Theorem \ref{Gra1}. If $\A$ is a prime algebra, then $\phi$ is a scalar
multiple of either an epimorphism or an antiepimorphism.
\end{corollary}

\begin{proof}
If $\A$ is prime, then
 epimorphisms or  antiepimorphisms are the only Jordan epimorphisms by Herstein's theorem  \cite{H}.
\end{proof}

\subsection{Polynomials satisfying \eqref{pd2}}
The treatment of \eqref{pd2} is similar to that of \eqref{pd1}.

\begin{lemma}\label{L32}
Let $f$ be a multilinear polynomial satisfying \eqref{pd2}. Suppose $\A$ is a unital algebra and   $\phi:\A\to \A$ is a linear map  preserving zeros of $f$ and satisfying
 $\phi(1)\in F^*1$. 
If $a,b\in A$ are such that $ab=ba=0$, then $\phi(a)\phi(b) = \phi(b)\phi(a)$.
  \end{lemma}

\begin{proof}
 We can reword \eqref{pd2} as
$$ \sum_{\sigma\in S_d}\lambda_{\sigma} = 0 \quad\mbox{and}\quad \mu:=\sum_{\sigma^{-1}(1) < \sigma^{-1}(2)}\lambda_\sigma \ne 0.$$
Therefore
$$\sum_{\sigma^{-1}(2) < \sigma^{-1}(1)}\lambda_\sigma =-\mu.$$
We may assume, for the same reason as in the proof of Lemma \ref{L31}, that $\phi(1)=1$. If $a,b\in \A$ are such that $ab=ba=0$, then
 $$f(a,b,1,\ldots,1) = 0,$$
  and hence $$f(\phi(a),\phi(b),1,\ldots,1) = 0.$$
    Since
$$ f(\phi(a),\phi(b),1,\ldots,1)
 = \sum_{\sigma^{-1}(1) < \sigma^{-1}(2)}\lambda_\sigma \phi(a)\phi(b) + \sum_{\sigma^{-1}(2) < \sigma^{-1}(1)}\lambda_\sigma \phi(b)\phi(a),$$
 it follows that
 $\mu\Bigl(\phi(a)\phi(b) - \phi(b)\phi(a) \Bigr)=0$, i.e., $\phi(a)$ and $\phi(b)$ commute. 
 \end{proof}

\begin{theorem} \label{Gra2}
Let $f$ be a multilinear polynomial of degree $d\ge 2$ satisfying \eqref{pd2}, and let
 $\A$ be a unital prime $C^*$-algebra that is not isomorphic to $M_2(\mathbb C)$.  If a continuous
bijective linear map $\phi:\A\to\A$ preserves zeros of $f$ and satisfies $\phi(1)\in \mathbb{C}^*\cdot 1$, then there exist  $\alpha\in \mathbb C$, an automorphism or an antiautomorphism $\theta$ of $\A$, and
a linear functional $f$ on $\A$ such that $\phi(a) = \alpha\theta(a) + f(a)1$
for all $a \in\A$.
\end{theorem}

\begin{proof}
 Lemma \ref{L32} makes it possible for us to apply \cite[Corollary 3.6]{ABEV}, which immediately gives the result.
\end{proof}

\subsection{Polynomials satisfying \eqref{pd3}} The condition \eqref{pd3} means that there exist $\lambda_1,\ldots,\lambda_d \in F$, not all zero, such that
$$
\sum_{i=1}^d \lambda_i = 0\quad\mbox{and}\quad f(x,x,\ldots,x,y) = \sum_{i=1}^d  \lambda_i x^{d-i}yx^{i-1}.
$$
The simplest case where $f=x_1x_2 - x_2x_1$ was considered in \cite[Theorem 2]{Br}. This result was one of the earliest applications of functional identities. Incidentally,  \cite[Theorem 2]{Br} was used  in the proof of  \cite[Corollary 3.6]{ABEV}, and therefore indirectly also in the proof of Theorem \ref{Gra2}. What we would now like to show is that using the advanced theory of functional identities one can  handle, in a more or less similar fashion,  a more general situation where $f$ satisfies \eqref{pd3}. 

Functional identities can be informally described as identical relations on rings that involve arbitrary (``unknown") functions. The goal is to describe these functions, or, when this is not possible, to determine the structure of the ring in question. 
For a full account on functional identities, as well as to some other notions that will appear below, we refer  to the book \cite{FI}.

\begin{theorem} \label{Gra3}
Let $f$ be a multilinear polynomial of degree $d\ge 2$ satisfying \eqref{pd3}, let {\rm char}$(F)\ne 2,3$, and let
 $\A$ be a centrally closed  prime $F$-algebra with  $\dim_F \A> d^2$.  If a
bijective linear map $\phi:\A\to\A$ preserves zeros of $f$, then there exist  $\alpha\in F$, an automorphism or an antiautomorphism $\theta$ of $\A$, and
a linear functional $f$ on $\A$ such that $\phi(a) = \alpha\theta(a) + f(a)1$
for all $a \in\A$.
\end{theorem}

\begin{proof}
As  $f(x,x,\ldots,x,x^2)$ is obviously $0$ if $f$ satisfies \eqref{pd3}, we have
$$ f(\phi(a),\phi(a),\ldots,\phi(a),\phi(a^2)) =0$$
for all $a\in\A$, i.e.,
$$\sum_{i=1}^d  \lambda_i \phi(a)^{d-i}\phi(a^2)\phi(a)^{i-1}=0.$$
A complete linearization of this identity leads to a situation where \cite[Theorem 4.13]{FI} is applicable under suitable assumptions on $\A$ and $\phi$.
In view of \cite[Theorems 5.11 and C.2]{FI}, these assumptions are fulfilled in our case since $\phi$ is surjective and $\dim_F \A> d^2$.
The conclusion is that
$\phi(ab +ba)$ is a quasi-polynomial. As char$(F)\ne 2$, this is equivalent to the existence of 
$\lambda\in F$ and maps $\mu,\nu:A\to F$ (with $\mu$ linear) such that
$$
\phi(a^2) = \lambda\phi(a)^2 + \mu(a)\phi(a) + \nu(a)
$$
for every $a\in \A$. Since $\phi$ is also injective and char$(F)\ne 3$, the result now follows from \cite[Theorem 2]{Br}.
\end{proof}

It is worth pointing out that all prime $C^*$-algebras are centrally closed \cite[Proposition 2.2.10]{AM}.
Let us also mention that infinite dimensional algebras are not exluded in Theorem \ref{Gra3};
 only algebras of ``small" dimension $\le d^2$ are. 

\bigskip

{\bf Acknowledgement.} The authors are  grateful to the referee for a careful reading of the
paper and comments which have improved the readability of the paper.

\end{document}